\documentclass[a4paper,11pt]{amsart}
\usepackage[english]{babel}

\newtheorem{theorem}{Theorem}[section]
\newtheorem{lm}[theorem]{Lemma}

\theoremstyle{definition}

\theoremstyle{remark}

\usepackage{amssymb}

\def\veps{\varepsilon}

\def\RR{\mathbb R}

\def\RN{{\mathbb R^N}}

\def\Om{\Omega}
\def\De{\Delta}
\def\Ga{\Gamma}

\def\ga{\gamma}

\def\om{\omega}
\def\pa{\partial}

\def\veps{\varepsilon}

\def\th{\theta}
\def\si{\sigma}

\def\dist{\mbox{\rm dist}}
\newcommand{\osc}{\mathop{\mathrm{osc}}}

\def\ovr{\overline}

\begin{document}
    \title[]{A note on Serrin's overdetermined problem}

\author[G. Ciraolo]{Giulio Ciraolo}
\address{Dipartimento di Matematica e
Informatica, Universit\`a di Palermo, Via Archirafi 34, 90123, Italy.}
\email{g.ciraolo@math.unipa.it}
\urladdr{http://www.math.unipa.it/~g.ciraolo/}

\author[R. Magnanini]{Rolando Magnanini}
    \address{Dipartimento di Matematica ed Informatica ``U.~Dini'',
Universit\` a di Firenze, viale Morgagni 67/A, 50134 Firenze, Italy.}
    \email{magnanin@math.unifi.it}
    \urladdr{http://web.math.unifi.it/users/magnanin}

\begin{abstract}
We consider the solution of the torsion problem 
$$
-\De u=1 \ \mbox{ in } \ \Om, \ \ u=0 \ \mbox{ on } \ \pa\Om.
$$
\par
Serrin's celebrated symmetry theorem states that, if the normal derivative $u_\nu$ is constant on $\pa\Om$, then $\Om$ must be a ball. In \cite{CMS2}, it has been conjectured that Serrin's theorem may be obtained {\it by stability} in the following way: first, for the solution $u$ of the torsion problem
prove the estimate 
$$
r_e-r_i\le C_t\,\Bigl(\max_{\Ga_t} u-\min_{\Ga_t} u\Bigr)
$$
for some constant $C_t$ depending on $t$, where $r_e$ and $r_i$ are the radii of an annulus containing $\pa\Om$ and 
$\Ga_t$ is a surface parallel to $\pa\Om$ at distance $t$ and sufficiently close to $\pa\Om$; secondly, if in addition  $u_\nu$ is constant on $\pa\Om$, show that
$$
\max_{\Ga_t} u-\min_{\Ga_t} u=o(C_t)\ \mbox{ as } \ t\to 0^+.
$$
\par
The estimate constructed in \cite{CMS2} is not sharp enough to achieve this goal. In this paper,
we analyse a simple case study and show that the scheme is successful if the admissible domains $\Om$ are ellipses.

%
\end{abstract}


\maketitle

\section[Introduction]{Introduction}
Let $\Om$ be a bounded domain in $\RN$ and let $u$ be the solution of the torsion problem
\begin{equation}\label{pb torsion}
-\De u=1 \ \mbox{ in } \ \Om, \ \ u=0 \ \mbox{ on } \ \pa\Om.
\end{equation}
Serrin's celebrated symmetry theorem \cite{Se} states that, if there exists a solution of \eqref{pb torsion} 
whose (exterior) normal derivative $u_\nu$ is constant on $\pa\Om$, that is such that
\begin{equation}\label{overdet serrin}
u_\nu= c \ \mbox{ on } \ \pa\Om,
\end{equation}
then $\Om$ is a ball and $u$ is radially symmetric.

As is well-known, the proof of Serrin makes use of the {\it method of moving planes} (see \cite{Se,Fr}), a refinement of Alexandrov's reflection principle \cite{Al}. 


The aim of this note is to probe the feasibility of a new proof of Serrin's symmetry theorem based on a comparison with another overdetermined problem for \eqref{pb torsion}.
In fact, it has been noticed that, under certain sufficient conditions on $\pa\Om$, if the solution of 
\eqref{pb torsion} is constant on a surface {\it parallel} to $\pa\Om$, that is, if for some small $t>0$
\begin{equation}\label{Gamma_t def}
u=k \ \mbox{ on } \ \Ga_t, \ \mbox{ where } \ \Gamma_t = \{x \in \Om:\ \dist(x,\pa \Om)=t\},
\end{equation}
then $\Om$ must be a ball  (see \cite{MSaihp, MSm2as, CMS2} and \cite{Sh}). 
\par
Condition \eqref{Gamma_t def} was first studied in \cite{MSaihp} (see also \cite{MSm2as} and \cite{CMS} for further developments), motivated by an investigation on {\it time-invariant level surfaces} of a nonlinear non-degenerate {\it fast diffusion} equation (tailored upon the heat equation), and was used
to extend to nonlinear equations the symmetry results obtained in \cite{MS1} for the heat equation.
The proof still hinges on the method of moving planes, that can be applied in a much simplified manner,
since the overdetermination in \eqref{Gamma_t def} takes place inside $\Om$. 
Under slightly different assumptions and by a different proof --- still based on the method of moving planes --- a similar result was obtained in \cite{Sh} independently.

\par
The evident similarity between the two problems arouses a natural question:
{\it is condition \eqref{Gamma_t def} weaker or stronger than \eqref{overdet serrin}? } 

As pointed out in \cite{CMS2}, 
\eqref{Gamma_t def} seems to be weaker than \eqref{overdet serrin}, as explained
by the following two observations: (i)  as \eqref{Gamma_t def} does not imply \eqref{overdet serrin}, the latter can be seen as the limit of a sequence of conditions of type \eqref{Gamma_t def} with $k=k_n$ and $t=t_n$ and
$k_n$ and $t_n$ vanishing as $n\to\infty$; (ii) as \eqref{overdet serrin} does not imply \eqref{Gamma_t def} either, if $u$ satisfies \eqref{pb torsion}--\eqref{overdet serrin}, then the oscillation of $u$ on a surface parallel to the boundary becomes smaller than usual, the closer the surface is to $\pa \Om$. 
More precisely, if $u\in C^1(\ovr{\Om})$, by a Taylor expansion argument, it is easy to verify that
\begin{equation} \label{oscillation o(t)}
\max_{\Gamma_t} u - \min_{\Gamma_t} u = o(t) \ \mbox{ as } \ t\to 0
\end{equation}
 --- that becomes a $O(t^2)$ as $t\to\infty$ when $u\in C^2(\ovr{\Om})$.
\par
This remark suggests the possibility that Serrin's symmetry result may be obtained {\it by stability} in the following way: first, for the solution $u$ of the torsion problem \eqref{pb torsion} prove the estimate 
\begin{equation}
\label{estimate}
r_e-r_i\le C_t\,\Bigl(\max_{\Ga_t} u-\min_{\Ga_t} u\Bigr)
\end{equation}
for some constant $C_t$ depending on $t$,  where $r_e$ and $r_i$ are the radii of an annulus containing $\pa\Om$; secondly, 
if in addition  $u_\nu$ is constant on $\pa\Om$, show that
$$
\max_{\Ga_t} u-\min_{\Ga_t} u=o(C_t)\ \mbox{ as } \ t\to 0^+.
$$
\par
In the same spirit of \eqref{estimate}, based on \cite{ABR}, in \cite{CMS2} we proved an estimate 
that quantifies the radial symmetry of $\Om$ in terms of the following quantity:
\begin{equation}\label{seminorm CMS2}
[u]_{\Ga_t}= \sup_{\substack{z, w\in \Ga_t \\ z \neq  w}} \frac{|u(z)-u( w)|}{|z- w|}.
\end{equation}
In fact, it was proved that there exist two constants $\veps, C_t>0$ such that, if
$[u]_{\Ga_t} \le\veps$,
then there are two concentric balls $B_{r_i}$ and $B_{r_e}$ such that
\begin{eqnarray}
B_{r_i}\subset\Om\subset B_{r_e}\quad\mbox{ and } \quad
r_e-r_i\le C_t\,[u]_{\Ga_t}  .\label{stabil2}
\end{eqnarray}
The constant $C_t$ only depends on $t$, $N$, the regularity
of $\pa\Om$ and the diameter of $\Om$. 
\par
The calculations in \cite{CMS2} imply that
$C_t$ blows-up exponentially as $t$ tends to $0$, which is too fast for our purposes,
since $[u]_{\Ga_t}$ cannot vanish faster than $t^2$, when \eqref{overdet serrin} holds. The exponential dependence of $C_t$ on $t$ is due to the method of proof we employed, which is based on the idea of
refining the method of moving planes from a quantitative point of view. As that method is
based on the maximum (or comparison) principle, its quantitative counterpart is based on {\it Harnack's inequality} and some quantitative versions of {\it Hopf's boundary lemma}. The exponential dependence of the constant involved in Harnack's inequality leads to that of $C_t$. Recent (unpublished) calculations,
based on more refined versions of Harnack's inequality, 
show that the growth rate of $C_t$ can be improved, but they are still inadequate to achieve our goal.
Approaches to stability based on the ideas contained in \cite{BNST2} and \cite{BNST3} do not seem to work for problem \eqref{pb torsion}-\eqref{Gamma_t def}.



In this note, we shall show that our scheme (i)-(ii) is successful, at least if the admissible 
domains are ellipses:
in this case, the deviation from radial symmetry can be exactly computed in terms of the oscillation of $u$ on $\Ga_t$. We obtain \eqref{estimate} with $C_t=O(t^{-1})$ as $t\to 0^+$;
thus, formula \eqref{oscillation o(t)} yields the desired symmetry.\footnote{Of course, in this very special case, there is a trivial proof of symmetry, but this is not the point.}

\setcounter{equation}{0}

\section{Section 2}
We begin by defining the three quantities that we shall exactly compute later on.  Let $\Ga$ be a $C^1$-regular closed simple plane curve and let $z(s)$, $s\in [0,|\Gamma|)$ be its parameterization by arc-lenght. For a function $u: \Gamma \to \RR$, we will consider 
the seminorms
\begin{equation}\label{sdf}
    |u|_\Ga= \sup_{\substack{0\le s, s'\le |\Ga|\\ s\neq s'}}\,\frac{|u(z(s))-u(z(s'))|}{\min(|s-s'|,|\Gamma|-|s-s'|) }, \ 
[u]_{\Ga}= \sup_{\substack{z, w\in \Ga \\ z \neq  w}} \frac{|u(z)-u( w)|}{|z- w|},
\end{equation}
and the {\it oscillation}
\begin{equation}\label{osc}
\osc_\Ga u=\max_\Ga u-\min_\Ga u.
\end{equation}

We now consider an ellipse
\begin{equation*} \label{Omega ellipse}
E=\{z=(x,y) \in \RR^2:\ \frac{x^2}{a^2}+\frac{y^2}{b^2} < 1\},
\end{equation*}
with semi-axes $a$ and $b$ normalized by $a^{-2}+b^{-2}=1$, and let
\begin{equation}\label{defGammat}
\Ga_t=\{z \in E:\, \dist(z,\pa E) = t\}
\end{equation}
be the curve parallel to $\pa E$ at distance $t$; $\Ga_t$ is still regular and simple if $t$ is smaller than the minimal radius of curvature of $\pa E$, that is for 
\begin{equation}
\label{radius}
0\le t<\frac{\min(a^{3}, b^{3})}{2 a^2 b^2}\,.
\end{equation}
The solution $u$ of \eqref{pb torsion} is clearly given by
\begin{equation}\label{u ellipse}
u(x,y)=1- \frac{x^2}{a^2} - \frac{y^2}{b^2}.
\end{equation}

\begin{lm}
\label{lm:conti}
Let $u$ be given by \eqref{u ellipse} and let $t$ satisfy \eqref{radius}. Then, we have:
\begin{enumerate}
\item[(i)] $\displaystyle |u|_{\Ga_t} = |a-b|\,\frac{a+b}{a^2 b^2}\,t$;
\item[(ii)] $\displaystyle [u]_{\Ga_t}=|u|_{\Ga_t}$;
\item[(iii)] $\displaystyle \osc_{\Ga_t} u=|a-b|\,\frac{a+b}{a^2 b^2}\,\Bigl(\frac{2ab}{a+ b}- t\Bigr)\,t$.
\end{enumerate}
\end{lm}

\begin{proof}
The standard parametrization of $\pa E$ is
$$
\ga(\th)=(a\cos\th,b\sin\th), \ \th\in[0,2\pi];
$$
thus,
$$
\Ga_t = \Big{\{}\gamma(\th)-t\,J\, \frac{\gamma'(\th)}{|\gamma'(\th)|} :\ \th \in [0,2\pi) \Big{\}},
$$
where $J$ is the rotation matrix
\begin{equation*}
\left(
\begin{array}{cc}
0 & 1  \\
-1 & 0
\end{array}
\right),
\end{equation*}
so that the outward unit normal is
$$
\nu(\th)=J\, \frac{\ga'(\th)}{|\ga'(\th)|}.
$$
\par
(i) The mean value theorem then tells us that
\begin{equation}\label{mean value thm}
\frac{|u(z(s))-u(z(s'))|}{\min(|s-s'|,|\Ga_t|-|s-s'|)} = |\langle D u(z(\sigma)),    z'(\sigma)\rangle|,
\end{equation}
for some $\sigma \in [0,|\Ga_t|]$. Since $\Ga_t$ is parallel to $\pa E$, we have
\begin{equation*}
 z'(\sigma)=\frac{\gamma'(\th(\sigma))}{|\gamma'(\th(\sigma))|},
\end{equation*}
where $\th(\sigma)$ is such that
\begin{equation*}
z(\sigma)= \gamma(\th(\sigma)) - t\, \nu(\th).
\end{equation*}
By \eqref{u ellipse}, we have that
\begin{equation*}
|\langle Du(z(\sigma)),  z'(\sigma)\rangle| = 2 | \langle Az(\sigma), z'(\sigma)\rangle| \ \mbox{ with } \ A = \left(
\begin{array}{cc}
a^{-2} & 0  \\
0 & b^{-2}
\end{array}
\right),
\end{equation*}
and hence
\begin{equation*}
|\langle Du(z(\sigma)),  z'(\sigma)\rangle|= 2 \Bigl| \frac{\langle A\gamma(\th),\ga'(\th)\rangle}{|\ga'(\th)|} - t\, \frac{\langle A\,J\,\ga'(\th),\gamma'(\th)\rangle }{|\gamma'(\th)|^2} \Bigr|,
\end{equation*}
with $\th=\th(\sigma)$. 
\par
Straightforward computations give:
\begin{eqnarray*}
&\gamma'(\th) = (-a\sin\th,b \cos \th),  &|\gamma'(\th)|= \sqrt{a^2\sin^2\th+b^2\cos^2\th},\\
&\langle A \gamma(\th),\gamma'(\th)\rangle = 0, 
&\langle AJ\gamma'(\th),\gamma'(\th)\rangle  = \frac{|a^2-b^2|}{ab}\, \sin\th \, \cos \th.
\end{eqnarray*}
Therefore,
\begin{equation*}
|\langle Du(z(\sigma)) \cdot  z'(\sigma)\rangle| = \frac{|a^2-b^2|}{ab}\,  \frac{2 |\tan \th|}{a^2 \tan^2\th + b^2}\,t;
\end{equation*}
this expression achieves its maximum if $|\tan\th| = b/a$, that gives:
\begin{equation*}
\max_{0\le\si\le |\Ga_t|}|\langle Du(z(\sigma)),  z'(\sigma)\rangle| = |a^{-2}-b^{-2}|\,t.
\end{equation*}
>From \eqref{mean value thm} we conclude.
\par
(ii) By a symmetry argument, we can always assume that $[u]_{\Ga_t}$ is attained for points $z$ and $ w$ (that may possibly coincide) in the first quadrant of the cartesian plane.
\par
Now, suppose that the value $[u]_{\Ga_t}$ is attained for two points $z, w\in\Ga_t$ with $z\not=w$.
Let $s\to z(s)\in\Ga_t$ be a parametrization by arclength of $\Ga_t$ such that $z(0)=z$ and let 
$\om= z'(0)$ be the tangent unit vector to $\Ga_t$ at $z$. The function defined by
$$
f(s)=\frac{u(z(s))-u(w)}{|z(s)-w|}
$$
has a relative maximum at $s=0$ and hence $f'(0)=0$; thus,
$$
\frac{\langle Du(z), \om\rangle}{|z-w|}=\frac{u(z)-u(w)}{|z-w|}\,\frac{\langle z-w, \om\rangle}{|z-w|^2}.
$$
Therefore, since $\langle z-w, \om\rangle\not=0$, we have that
$$
[u]_{\Ga_t}=\frac{\langle Du(z), \om\rangle}{\langle z-w, \om\rangle}\,|z-w|,
$$
that gives a contradiction, since the right-hand side increases with $z$ if the angle between $z-w$ and 
$\om$ decreases. 
\par
As a consequence,  we infer that
$$
[u]_{\Ga_t}=\lim_{n\to\infty}\frac{u(z_n)-u(w_n)}{|z_n-w_n|} \ \mbox{ where } \  z_n, w_n\in\Ga_t
\ \mbox{ and } \ |z_n-w_n|\to 0.
$$ 
Thus, by compactness, we can find a point $z\in\Ga_t$ such that
$$
[u]_{\Ga_t}=\langle Du(z), \om\rangle,
$$
where $\om$ is the tangent unit vector to $\Ga_t$ at $z$.
\par
It is clear now that $[u]_{\Ga_t}=|u|_{\Ga_t}$.
\vskip.2cm
(iii) If \eqref{radius} holds, the maximum and minimum of $u$ on $\Ga_t$ are attained at the points on $\Ga_t$ whose projections on $\pa E$ respectively maximize and minimize $|Du|$ on $\pa E$.
Thus, (iii) follows at once.
\par
In fact, for a point $z=\ga(\th)-t\,\nu(\th)$ on $\Ga_t$, calculations give that
\begin{eqnarray*}
u(z)&=&1-\langle A\,\ga(\th), \ga(\th)\rangle+2t\,\langle A\,\ga(\th), \nu(\th)\rangle-
t^2 \langle A\,\nu(\th), \nu(\th)\rangle=\\
&&2t\,\langle A\,\ga(\th), \nu(\th)\rangle-
t^2 \langle A\,\nu(\th), \nu(\th)\rangle,
\end{eqnarray*}
where
\begin{eqnarray*}
&&\langle A\,\ga(\th), \nu(\th)\rangle=\frac1{ab}\,\sqrt{b^2\cos^2\th+a^2\sin^2\th}\,;\\
&&\langle A\,\nu(\th), \nu(\th)\rangle=\frac1{a^2 b^2}\,\frac{b^4\cos^2\th+a^4\sin^2\th}{b^2\cos^2\th+a^2\sin^2\th}\,,
\end{eqnarray*}
so that, by the substitution $\xi=\sqrt{b^2\cos^2\th+a^2\sin^2\th}$, we obtain that
$$
u(z)=\frac{2t}{ab}\,\xi+\frac{t^2}{\xi^2}-(a^{-2}+b^{-2})\,t^2.
$$
Since \eqref{radius} holds, this function is respectively maximal or minimal when 
$\xi=\min(a,b)$ or $\max(a,b)$.
 \end{proof}


Therefore, for an ellipse $E$, \cite{CMS2}[Theorem 1.1] can be stated as follows, together with two analogues.

\begin{theorem} \label{thm seminorm stability}
Let $u$ be the solution of \eqref{pb torsion} in an ellipse $E$ of semi-axes $a$ and $b$. 
Let $\Ga_t$ be the curve \eqref{defGammat} parallel to $\pa E$ at distance $t$ satisfying \eqref{radius}.
\par
Then, there are two concentric balls $B_{r_i}$ and $B_{r_e}$ such that $ B_{r_i}\subset E\subset B_{r_e}$ and
\begin{equation*}\label{stability ellipse}
r_e-r_i =  \frac{1}{t}\, \frac{a^2b^2}{a+b}\,|u|_{\Ga_t}; \ \ r_e-r_i =  \frac{1}{t}\, \frac{a^2b^2}{a+b}\,[u]_{\Ga_t};
\end{equation*}
\begin{equation*}\label{stability ellipse2}
r_e-r_i =  \frac{1}{t}\, \frac{a^2b^2}{a+b}\,\osc_{\Ga_t}u.
\end{equation*}
\end{theorem}

\begin{proof}
The largest ball contained in $E$ and the smallest ball containing $E$ are centered at the origin and have
radii $\min(a,b)$ and $\max(a,b)$, respectively; hence, $r_e-r_i=|a-b|$ and the desired formulas follow from Lemma \ref{lm:conti}.
\end{proof}


Now, we turn to Serrin problem \eqref{pb torsion}-\eqref{overdet serrin}. The following lemma 
holds for quite general domains in general dimension .

\begin{lm} \label{lemma seminorm asympt}
Let $\Om\subset\RN$ be a bounded domain with boundary of class $C^2$ and let $u\in  C^1(\ovr{\Om})\cap C^2(\Om)$ be a solution of \eqref{pb torsion} satisfying \eqref{overdet serrin}. 
\par
Then
$$
\osc_{\Ga_t}u=o(t) \ \mbox{ as } \ t\to 0^+.
$$
If, in addition, $u\in C^2(\ovr{\Om})$,
then
\begin{equation*}\label{seminorm asympt}
[u]_{\Ga_t} \ \mbox{ and } \ |u|_{\Ga_t}= o(t) \ \mbox{ as } \ t\to 0^+.
\end{equation*}

\end{lm}

\begin{proof}
Let $z$ and $z'\in\Ga_t$ points at which $u$ attains its maximum and minimum, respectively 
(for notational semplicity, we do not indicate their dependence on $t$). If $t$ is sufficiently small, 
they have unique projections, say $\ga$ and $\ga'$, on $\pa\Om$, so that we can write that
$z=\ga-t\,\nu(\ga)$ and $z=\ga'-t\,\nu(\ga')$.
\par
Since both $u$ and $u_\nu$ are constant on $\pa\Om$,  Taylor's formula gives:
$$
u(z)-u(z')=\int_0^t [\langle Du(\ga'-\tau\,\nu(\ga')),\nu(\ga')\rangle-\langle Du(\ga-\tau\,\nu(\ga)),\nu(\ga)\rangle]\,d\tau.
$$
By the (uniform) continuity of the first derivatives of $u$ (and the normals), the right-hand side of the
last identity is a $o(t)$ as $t\to 0^+$. 
\par
We shall prove the second part of the theorem only for the semi-norm $[u]_{\Ga_t} $, since that for $|u|_{\Ga_t}$ runs similarly.   
\par
Let $s$ and $s'\in [0, |\Ga_t|]$ attain the first supremum in \eqref{sdf}; 
we apply \eqref{mean value thm} and obtain that
\begin{equation*}
\frac{|u(z(s))-u(z(s'))|}{\min(|s-s'|,|\Ga_t|-|s-s'|)} = |\langle Du(z(\si)), z'(\si)\rangle|,
\end{equation*}
for some $\sigma \in [0,|\Ga_t|)$. Let $\ga\in \pa \Om$ be the projection of
the point $z=z(\si)$ on $\pa\Om$, that is
$
z=\ga-t\, \nu(\ga).
$
\par
Since $\pa\Om$ and $\Ga_t$ are parallel,  the tangent unit vector $\tau(\ga)$ to the curve $\si\mapsto\ga(\si)\in\pa\Om$ at $\ga$ equals the tangent unit vector $\tau(z)$ to the curve $\si\mapsto z(\si)\in\Ga_t$ at $z$; the same occurs for the corresponding normal unit vectors $\nu(\ga)$ and $\nu(z)$.  
\par
It is clear that $\langle Du(\ga), \tau(\ga)\rangle=0$ and, since $u\in C^2(\ovr{\Om})$, by differentiating \eqref{overdet serrin}, we also have that 
$
\langle D^2u(\ga)\, \nu(\ga), \tau(\ga) \rangle=0;
$
thus, by Taylor's formula, we obtain that
\begin{eqnarray*}
\langle Du(z(\si)), z'(\si)\rangle\!\!\!&=&\!\!\!\langle Du(\ga), \tau(\ga)\rangle- t\,\langle D^2u(\ga)\, \nu(\ga), \tau(\ga) \rangle + R(s,s',t)=\\
&&\!\!\!R(s,s',t).
\end{eqnarray*}
Since the second derivatives of $u$ are uniformly continuous on $\ovr{\Om}$, we conclude that the
remainder term $R(s,s',t)$ is a $o(t)$ as $t\to 0^+$.
\end{proof}

\begin{theorem}
Let $E$ be an ellipse of semi-axes $a$ and $b$ and assume that in $E$ there exists a solution $u$ of \eqref{pb torsion} satisfying \eqref{overdet serrin}. 
\par
Then $a=b$, that is $E$ is a ball and $u$ is radially symmetric.
\end{theorem}
\begin{proof}
Theorem \ref{thm seminorm stability} and Lemma \ref{lemma seminorm asympt} in any case yield that
\begin{equation*}
|a-b| = o(1) \ \mbox{ as } \ t\to 0^+,
\end{equation*}
 which implies the assertion.
\end{proof}

\end{document}